\documentclass[leqno,11pt]{amsart}
\usepackage{amsfonts} %
\usepackage[backref=page]{hyperref}
\usepackage{amsmath,amssymb} %
\usepackage{latexsym}
\numberwithin {equation}{section} %
\newtheorem{theo}{Theorem}[section] %
\newtheorem{prop}[theo]{Proposition} %

\newtheorem{remark}[theo]{Remark}
\theoremstyle{definition}

\newtheorem{example}[theo]{Example}

\newtheorem{defi}[theo]{Definition}
\newcommand{\clos}{\operatorname{cl}} %
\newcommand{\diam}{\operatorname{diam}} %

\begin{document}

\title{Completeness in quasi-pseudometric spaces -- a survey}
\author{S. Cobza\c{s} }
\address{\it Babe\c s-Bolyai University, Faculty of Mathematics
and Computer Science, 400 084 Cluj-Napoca, Romania}\;\;
\email{scobzas@math.ubbcluj.ro}
\thanks{Published in  \textbf{Mathematics 2020, 8, 1279}. }
\begin{abstract}

The aim of this paper  is to discuss the relations between various notions of sequential completeness and the corresponding notions of completeness by nets or by filters in the setting of quasi-metric spaces. We propose a new definition of right $K$-Cauchy net in a quasi-metric  space for which the  corresponding completeness is equivalent to the sequential completeness. In this way we complete some results of R.~A. Stoltenberg, Proc. London Math. Soc. \textbf{17} (1967), 226--240,  and V.~Gregori and  J.~Ferrer, Proc. Lond. Math. Soc., III Ser., \textbf{49} (1984), 36. A discussion on nets defined over ordered or pre-ordered directed sets is also included.

\medskip
\textbf{Classification MSC 2020:}  54E15 54E25 54E50 46S99


\end{abstract}
\date{\today}
\maketitle

\section{Introduction}

It is well known that completeness is an essential tool in the study of  metric spaces, particularly for fixed points results in such spaces.
 The study of completeness in quasi-metric spaces is considerably more involved, due to the lack of symmetry of the distance -- there are several notions of completeness all agreing with the usual one in the metric case  (see \cite{reily-subram82} or \cite{Cobzas}).  Again these notions are essential in proving fixed point results in quasi-metric spaces as it is shown by some papers on this topic as, for instance, \cite{chandok20}, \cite{gupta20}, \cite{latif18}, \cite{romag-tirado15} (see also  the book \cite{AKR15}).  A survey on the relations between completeness and the existence of fixed points in various circumstances  is given in  \cite{cob18}.

It is known that in the metric case the notions of completeness by sequences, by nets and by filters all agree and, further, the completeness of a metric space is equivalent to the completeness of the associated uniform space (in all senses). In the present paper we show that, in the quasi-metric case,  these notions agree  in some cases, but there are situations when they are different, mainly for the notions of right-completeness.  Such a situation was emphasized by Stoltenberg \cite{stolt67}, who proposed a  notion of Cauchy net for which completeness agrees with the sequential completeness. Later Gregori and Ferrer \cite{greg-fer84} found a gap in the proof given by Stoltenberg and proposed a new version of right $K$-Cauchy net for which  sequential completeness and  completeness by nets agree. In the present paper we include a discussion on this Gregori-Ferrer notion of Cauchy net and complete these results by proposing a  notion of right-$K$-Cauchy    net for which the equivalence with sequential completeness holds.

\section{Metric and uniform spaces}

For a mapping $d:X\times X\to\mathbb{R}$ on a set $X$ consider the following conditions:
\begin{align*}
\mbox{(M1)}&\qquad d(x,y)\geqslant 0 \quad \mbox{and}  \quad d(x,x)=0;\\
\mbox{(M2)}&\qquad d(x,y)=d(y,x);\\
\mbox{(M3)}&\qquad d(x,z)\leqslant d(x,y)+d(y,z);\\
\mbox{(M4)}&\qquad d(x,y)=0 \;\Rightarrow\; x=y;\\
\mbox{(M4$'$)}&\qquad d(x,y)=d(y,x)=0\;\Rightarrow\; x=y,
\end{align*}
for all $x,y,z\in X.$

The mapping $d$ is called a  \emph{pseudometric} if it satisfies  (M1), (M2) and (M3). A pseudometric that also satisfies (M4) is called a \emph{metric}.

The \emph{open} and \emph{closed balls} in a pseudometric space $(X,d)$ are defined by
\begin{equation}\label{def.ball}
B_d(x,r)=\{y\in X: d(x,y)<r\} \quad\mbox{ and}\quad B_d[x,r]=\{y\in X: d(x,y)\leqslant r\}\,,
\end{equation}

respectively.

A \emph{filter} on a set $X$ is a nonempty family $\mathcal F$ of nonempty subsets of $X$  satisfying the conditions
\begin{align*}
  \mbox{(F1)}&\qquad F\subseteq G\;\mbox{ and }\;F\in\mathcal F\;\Rightarrow\; G\in \mathcal F;\\
  \mbox{(F2)}&\qquad F\cap G\in \mathcal F\quad\mbox{for all}\quad F,G\in \mathcal F.
\end{align*}

It is obvious that  (F2) implies
$$\mbox{(F2$'$)}\qquad F_1,\ldots ,F_n\in {\mathcal F}\;\Rightarrow \; F_1\cap \ldots \cap F_n\in {\mathcal F}.$$
for all $n\in\mathbb{N}$  and $F_1,\ldots ,F_n\in {\mathcal F}$.

A \emph{base} of a filter $\mathcal F$ is a subset $\mathcal B$ of $ \mathcal F$ such that every $F\in\mathcal F$ contains a $B\in\mathcal B$.

A nonempty family $\mathcal B$ of nonempty subsets of $X$ such that
$$
  {\rm (BF1)}\qquad \forall B_1,B_2\in \mathcal B,\; \exists B\in\mathcal B,\; B\subseteq B_1\cap B_2\,.
$$
  generates  a filter $\mathcal F(\mathcal B)$ given by
$$
\mathcal F(\mathcal B)=\{U\subseteq X :\exists\, B\in\mathcal B,\; B\subseteq U\}\,.$$

A family $\mathcal B$ satisfying (BF1) is called a \emph{filter base}.

A relation $\le$ on a set $X$ is called a \emph{preorder} if it is reflexive and transitive, that is,
$$
{\rm (O1)}\quad x\le x\quad\mbox{and}\quad {\rm (O2)}\quad x\le y\wedge y\le z\;\Rightarrow\; x\le z\,,$$
for all $x,y,z\in X$. If the relation $\le $ is further antireflexive, i.e.
$$
 {\rm (O3)}\quad x\le y\wedge y\le x\;\Rightarrow\; x= y\,,$$
 for all $x,y\in X$, then $\le $ is called an \emph{order}.  A set $X$ equipped with a preorder (order) $\le$  is called a preordered (ordered) set,
denoted as $(X,\le)$.

A preordered set $(I,\le)$ is called \emph{directed} if for every $i_1,i_2\in I$  there exists $j\in I$ with $i_1\le j$ and $i_2\le j$.
A \emph{net} in a  set $X$ is a mapping $\phi:I\to X$, where $(I,\le)$ is a directed set. The alternative notation $(x_i:i\in I)$, where $x_i=\phi(i),\, i\in I,$ is also used.

A \emph{uniformity} on a set $X$ is a filter $\mathcal U$ on $X\times X$ such that
\begin{equation*}
\begin{aligned}
\mbox{(U1)}&\qquad \Delta(X)\subseteq U,\; \forall U\in \mathcal{U};\\
\mbox{(U2)}&\qquad \forall U\in \mathcal{U},\; \exists V\in \mathcal{U},\;\mbox{ such that } \; V\circ V\subseteq U\;,\\
\mbox{(U3)}&\qquad \forall U\in \mathcal{U},\; U^{-1}\in \mathcal U.
\end{aligned}
\end{equation*}
where
\begin{align*}
 \Delta(X)&=\{(x,x) : x\in X\}\;\mbox{ denotes the diagonal of }\; X,\\
 M\circ N &=\{(x,z)\in X\times X : \exists y\in X,\; (x,y)\in
M\;\mbox{and} \; (y,z)\in N\}, \;\mbox{ and}\\
 M^{-1}&=\{(y,x) : (x,y)\in M\}\,,
\end{align*}
for any $M,N\subseteq X\times X.$

The sets in $\mathcal U$ are called \emph{entourages.}  A \emph{base} for a uniformity $\mathcal U$ is a base of the filter $\mathcal U$.
The composition $V\circ V$ is denoted sometimes simply by $V^2.$  Since every entourage contains the
diagonal $\Delta(X),$ the inclusion $V^2\subseteq U$ implies $V\subseteq U.$

For $U\in \mathcal{U}, \, x\in X$ and $Z\subseteq X$ put
$$
U(x) =\{y\in X: (x,y)\in U\}\quad\mbox{and}\quad U[Z]=\bigcup\{U(z):
z\in Z\}\,.
$$
 A uniformity $\mathcal{U}$ generates a topology $\tau(\mathcal{U})$
 on $X$ for which the family of sets
 $
\{ U(x): U\in \mathcal{U}\}$
is a base of neighborhoods of any point $x\in X.$

Let $(X,d)$ be a pseudometric space.  Then the pseudometric $d$ generates a topology  $\tau_d$ for which
$
\{B_d(x,r) : r>0\},$
is a base of neighborhoods for every $x\in X$.

The pseudometric $d$ generates also   a uniform structure $\mathcal U_d$ on $X$ having as basis of entourages
the sets
$$
U_\varepsilon=\{(x,y)\in X\times X: d(x,y)<\varepsilon\},\;\; \varepsilon>0\,.$$

Since
$$
U_\varepsilon(x)=B_d(x,\varepsilon),\; x\in X,\;\varepsilon>0,$$
it follows that the topology $\tau(\mathcal U_d)$ agrees with the topology $\tau_d$ generated by the pseudometric $d$.

A sequence $(x_n)$ in $(X,d)$ is called \emph{Cauchy} (or \emph{fundamental}) if for every $\varepsilon>0$ there exists $n_\varepsilon\in\mathbb{N}$ such that
$$
d(x_n,x_m)<\varepsilon \;\mbox{ for all }\; m,n\in \mathbb{N}\;\mbox{ with }\; m,n\geqslant n_\varepsilon\,,$$
a condition written also as
$$
\lim_{m,n\to\infty}d(x_m,x_n)=0\,.$$

A sequence $(x_n)$ in a uniform space $(X,\mathcal U)$ is called $\mathcal U$-\emph{Cauchy} (or simply \emph{Cauchy}) if for every $U\in\mathcal U$ there exists $n_\varepsilon\in \mathbb{N}$ such that
$$
(x_m,x_n)\in U  \;\mbox{ for all }\; m,n\in \mathbb{N}\;\mbox{ with }\; m,n\geqslant n_\varepsilon\,.$$

It is obvious that, in the case of a pseudometric space $(X,d)$, a sequence is Cauchy with respect to the pseudometric $d$ if and only if it is Cauchy with respect to the uniformity $\mathcal U_d$.

The Cauchyness of nets in pseudometric and in uniform spaces is defined by analogy with that of sequences.

A filter $\mathcal F$ in a uniform space $(X,\mathcal U)$ is called $\mathcal U$-\emph{Cauchy} (or simply \emph{Cauchy}) if for every $U\in\mathcal U$ there exists $F\in\mathcal F$ such that
$$
F\times F\subseteq U\,.$$

\begin{defi}
 A pseudometric space $(X,d)$ is called \emph{complete} if every Cauchy sequence in $X$ converges. A uniform space  $(X,\mathcal U)$  is called \emph{sequentially complete}
 if every $\mathcal U$-Cauchy sequence in $X$ converges and \emph{complete} if every $\mathcal U$-Cauchy net  in $X$ converges (or, equivalently, if every $\mathcal U$-Cauchy filter in $X$ converges).
\end{defi}
\begin{remark}
 We can define the completeness of a subset $Y$ of a pseudometric space  $(X,d)$ by the condition that
every Cauchy sequence in $Y$ converges to  some element of $Y$. A closed subset of a pseudometric space is complete and a complete subset of a metric space is closed.  A complete subset of a pseudometric space need not be  closed.
\end{remark}

The following result holds in the metric case.
\begin{theo} For a pseudometric space  $(X,d)$ the following conditions are equivalent.
\begin{enumerate}
\item[\rm 1.] The metric space $X$ is complete.
\item[\rm 2.]  Every Cauchy net in $X$ is convergent.
\item[\rm 3.] Every Cauchy filter in $(X,\mathcal U_d)$ is convergent.
\end{enumerate}
\end{theo}

An important result in metric  spaces is Cantor characterization of completeness.
\begin{theo}[Cantor theorem]\label{t.Cantor} A pseudometric space $(X,d)$ is complete if and only if  every descending sequence of nonempty closed subsets of $X$ with diameters tending to zero has nonempty intersection. This means that for any family $ F_n,\, n\in \mathbb{N},$   of nonempty closed subsets of $X$
$$
F_1\supseteq F_2 \supseteq\dots\quad\mbox{and}\quad \lim_{n\to \infty}\diam (F_n)=0\;\mbox{ implies } \; \bigcap_{n=1}^\infty F_n\ne \emptyset\,.$$

If $d$ is a metric then this intersection contains exactly one point.
\end{theo}

The \emph{diameter} of a subset $Y$ of a pseudometric space $(X,d)$ is defined by
\begin{equation}\label{def.diam}
\diam(Y)=\sup\{d(x,y) :x,y\in Y\}\,.\end{equation}

\section{Quasi-pseudometric and quasi-uniform spaces}
In this section we present the basic results on quasi-metric and quasi-uniform spaces needed in the sequel, using as basic source the book \cite{Cobzas}.

\subsection{Quasi-pseudometric   spaces}

Dropping the symmetry condition (M2) in the definition of a metric one obtains the notion of quasi-pseudometric, that is, a   {\it quasi-pseudometric} on an arbitrary set $X$ is a mapping $d: X\times X\to \mathbb{R}$ satisfying the  conditions (M1)  and (M3). If $d$ satisfies further (M4$'$) then it called a \emph{quasi-metric}.
The pair $(X,d)$ is called a {\it
quasi-pseudometric space}, respectively  a {\it quasi-metric space}\footnote{In \cite{Cobzas} the term ``quasi-semimetric" is used instead of ``quasi-pseudometric"}.
Quasi-pseudometric spaces were introduced by Wilson \cite{wilson31}.

The conjugate of the quasi-pseudometric
$d$ is the quasi-pseudometric $\bar d(x,y)=d(y,x),\, x,y\in X.$ The mapping $
d^s(x,y)=\max\{d(x,y),\bar d(x,y)\},\,$ $ x,y\in X,$ is a pseudometric on $X$ which is a metric if and
only if $d$ is a quasi-metric.

If $(X,d)$ is a quasi-pseudometric space, then for $x\in X$ and $r>0$ we define the balls in $X$ as in \eqref{def.ball}

 The topology $\tau_d$ (or $\tau(d)$) of a quasi-pseudometric space $(X,d)$ can be defined through  the family
$\mathcal{V}_d(x)$ of neighborhoods  of an arbitrary  point $x\in X$:%
\begin{equation*}
\begin{aligned}
V\in \mathcal{V}_d(x)\;&\iff \; \exists r>0\;\mbox{such that}\; B_d(x,r)\subseteq V\\
                             &\iff \; \exists r'>0\;\mbox{such that}\; B_d[x,r']\subseteq V. %
\end{aligned} %
\end{equation*}

The topological notions corresponding to $d$ will be prefixed by $d$- (e.g. $d$-closure, $d$-open, etc).

The convergence of a sequence $(x_n)$ to $x$ with respect to $\tau_d,$ called $d$-convergence and
denoted by
$x_n\xrightarrow{d}x,$ can be characterized in the following way %
\begin{equation}\label{char-rho-conv1} %
         x_n\xrightarrow{d}x\;\iff\; d(x,x_n)\to 0. %
\end{equation} %

Also
\begin{equation}\label{char-rho-conv2} %
         x_n\xrightarrow{\bar d}x\;\iff\;\bar d(x,x_n)\to 0\; \iff\; d(x_n,x)\to 0. %
\end{equation}

As a space equipped with two topologies, $\tau_d$ and   $\tau_{\bar d}\,$, a quasi-pseudometric space can be viewed as a bitopological space in the sense of Kelly \cite{kelly63}.

The following example of  quasi-metric space is an important source of counterexamples in topology (see  \cite{Steen-Seb}).

\begin{example}[The Sorgenfrey line]\label{ex.Sorgenfrey}
   For $x,y\in \mathbb{R} $ define a quasi-metric $d$ by $d(x,y)=y-x,\, $ if $x\leq y$ and $d(x,y)=1$
 if $x>y.$   A basis of open $\tau_d$-open neighborhoods of a point $x\in \mathbb{R}$ is formed by the family
 $[x,x+\varepsilon),\, 0<\varepsilon <1.$ The family of intervals $(x-\varepsilon,x],\, 0<\varepsilon<1,\,$ forms a basis of open
 $\tau_{\bar d}\,$-open neighborhoods of $x.$  Obviously, the topologies $\tau_d$ and $\tau_{\bar d}$ are
 Hausdorff and $d^s(x,y)=1$ for $x\neq y,$ so that \ $\tau(d^s)$ is the discrete topology of $\mathbb{R}.$
    \end{example}

\textbf{    Asymmetric normed spaces}

Let $X$ be a real vector space. A mapping $p:X\to\mathbb{R}$ is called an \emph{asymmetric seminorm} on $X$ if

\begin{align*}
\mbox{(AN1)}&\qquad p(x)\geqslant 0;\\
\mbox{(AN2)}&\qquad  p(\alpha x)=\alpha p(x);\\
\mbox{(AN3)}&\qquad  p(x+y) \leqslant p(x)+ p(y)\,,
\end{align*}
for all $x,y\in X$ and $\alpha\geqslant 0.$

If, further,   \begin{align*}
 \mbox{(AN4)}&\qquad   p(x)= p(-x) =0\;\Rightarrow\; x=0\,,
\end{align*}
for all $x\in X,$ then $p$ is called an \emph{asymmetric norm}.

To an asymmetric seminorm $p$ one associates a quasi-pseudometric $d_{p}$ given by
$$
d_{p}(x,y)=p(y-x),\quad x,y\in X,$$
which is a quasi-metric if $p$ is an asymmetric norm.
All the topological and metric notions in an asymmetric normed space are understood as those corresponding to this quasi-pseudometric  $d_{p}$ (see \cite{Cobzas}).

The following asymmetric norm on $\mathbb{R}$ is essential in the study of asymmetric normed spaces (see \cite{Cobzas}).

\begin{example}\label{ex.u-topR}
On the field $\mathbb{R} $ of  real numbers consider the asymmetric norm
$u(\alpha)=\alpha^+:=\max\{\alpha,0\}.$
Then, for  $\alpha\in \mathbb{R},\, \bar u(\alpha)=\alpha^-:=\max\{-\alpha,0\}$ and $u^s(\alpha)=|\alpha|.$
The topology $\tau(u)$ generated by $u$ is called the \emph{upper topology } of $\mathbb{R},$ while the topology
$\tau(\bar u)$ generated by $\bar u$ is called the \emph{lower topology} of $\mathbb{R}.$ A basis of open
$\tau(u)$-neighborhoods of a point $\alpha\in \mathbb{R}$ is formed of the intervals
$(-\infty,\alpha+\varepsilon),\,\varepsilon > 0.$ A basis of open $\tau(\bar u)$-neighborhoods  is formed of the intervals
$(\alpha-\varepsilon,\infty),\, \varepsilon > 0.$

 In this space the addition is continuous from $(\mathbb{R}\times\mathbb{R},\tau_u\times\tau_u)$ to $(\mathbb{R},\tau_u)$, but the multiplication
is discontinuous   at every point  $(\alpha,\beta)\in \mathbb{R}\times\mathbb{R}.$

 The multiplication is continuous from $(\mathbb{R}_+,|\cdot|)\times(\mathbb{R},\tau_u)$ to  $(\mathbb{R},\tau_u)$ but discontinuous at every point $(\alpha,\beta)\in (-\infty,0)\times\mathbb{R}$ to  $(\mathbb{R},\tau_u)$, when $(-\infty,0) $ is equipped with the topology generated by $|\cdot|$ and $\mathbb{R} $ with $\tau_u$.
 \end{example}

The following   topological properties are true for  quasi-pseudometric spaces.
    \begin{prop}[see \cite{Cobzas}]\label{p.top-qsm1}
   If $(X,d)$ is a quasi-pseudometric space, then the following hold.
   \begin{enumerate}
   \item[\rm 1.] The ball $B_d(x,r)$ is $d$-open and  the ball $B_d[x,r]$ is
       ${\bar{d}}$-closed. The ball    $B_d[x,r]$ need not be $d$-closed.
     \item[\rm 2.]
   The topology $d$ is $T_0$ if and only if  $d $ is a quasi-metric.   \\ The topology $d$ is $T_1$ if and only if
   $d(x,y)>0$ for all  $x\neq y$  in $X$.
      \item[\rm 3.]     For every fixed $x\in X,$ the mapping $d(x,\cdot):X\to (\mathbb{R},|\cdot|)$ is
   $d$-upper semi-continuous and ${\bar d}$-lower semi-continuous. \\
   For every fixed $y\in X,$ the mapping $d(\cdot,y):X\to (\mathbb{R},|\cdot|)$ is $d$-lower semi-continuous and
   ${\bar d}$-upper semi-continuous.
\end{enumerate}%
     \end{prop}

     Recall that a topological space $(X,\tau)$  is called:

 \begin{itemize}
 \item  $T_0$ if, for every pair of distinct points  in $X$, at least one of them has a neighborhood   not containing the other;
   \item  $T_1$ if, for   every pair of distinct points in $X$,  each of them has a neighborhood   not containing the other;
\item    $T_2$ (or {\it Hausdorff}) if  every  two distinct points  in $X$ admit  disjoint  neighborhoods;
\item  {\it regular}  if, for every   point $x\in X$ and closed set $A$ not containing $x$, there exist the disjoint open sets $U,V$ such that $x\in U$  and $A\subseteq V.$  \end{itemize}

\begin{remark} It is known that the topology $\tau_d$ of a pseudometric space $(X,d)$ is Hausdorff (or $T_2$) if and only if $d$ is a metric  if and only if any sequence in $X$ has at most one limit.

  The characterization of Hausdorff property  of quasi-pseudometric spaces can also be given in terms of uniqueness of the limits of sequences, as in the metric case: \ the topology of a quasi-pseudometric space $(X,d)$ is Hausdorff if and only if every sequence in $X$ has at most one $d$-limit if and only if every sequence in $X$ has at most one $\bar d$-limit (see \cite{wilson31}).

   In the case of an asymmetric seminormed space   there exists a characterization in terms of the asymmetric seminorm (see \cite{Cobzas}, Proposition 1.1.40).
  \end{remark}

\subsection{Quasi-uniform spaces}

Again,  the notion of quasi-uniform space is obtained by  dropping the symmetry condition (U3) from the definition of a uniform space, that is, a \emph{quasi-uniformity}
on a set $X$ is a filter $\mathcal U$ in $X\times X$ satisfying the conditions (U1) and (U2).
 The sets in
$\mathcal{U}$ are called \emph{entourages} and    the pair $(X,\mathcal{U})$ is called
a \emph{quasi-uniform space}, as in the case of  uniform spaces.

As uniformities, a quasi-uniformity $\mathcal{U}$ generates a topology $\tau(\mathcal{U})$
 on $X$ in a similar way: the sets
 \begin{equation*}
\{ U(x): U\in \mathcal{U}\}
\end{equation*}
form a base of neighborhoods of any point $x\in X.$

The topology $\tau(\mathcal U)$ is $T_0$ if and only if $\,\bigcap\mathcal U\,$ is a  partial order on $X$, and  $T_1$ if and only if $\,\bigcap\mathcal U=\Delta(X)$.

The family of sets
\begin{equation}\label{def.conj-qu}
\mathcal{U}^{-1}=\{U^{-1} : U\in\mathcal{U}\}
\end{equation}
is another quasi-uniformity on $X$ called the \emph{quasi-uniformity conjugate}  to $\mathcal U$. Also $\mathcal U\cup\mathcal U^{-1}$ is a subbase of a uniformity $\mathcal U^s$
on $X$, called the associated uniformity to the quasi-uniformity $\mathcal U$. It is the coarsest uniformity on $X$ finer than both $\mathcal U$ and $\mathcal U^{-1},\; \mathcal U^s=\mathcal U\vee\mathcal U^{-1}.$ A basis for $\mathcal U^s$ is formed by the sets $\{U\cap U^{-1}:U\in\mathcal U\}.$

If $(X,d)$ is a quasi-pseudometric space, then
\begin{equation*}
 U_\varepsilon =\{(x,y)\in X\times X : d(x,y)< \varepsilon\},\;
 \varepsilon > 0\;,
 \end{equation*}
 is a basis for a quasi-uniformity $\mathcal U_d$ on $X.$  The
 family
 \begin{equation*}
 U^-_\varepsilon =\{(x,y)\in X\times X : d(x,y)\leqslant \varepsilon\},\; \varepsilon > 0\;,
 \end{equation*}
 generates the same quasi-uniformity. Since $U_\varepsilon(x)=B_d(x,\varepsilon)$ and
 $U^-_\varepsilon(x)=B_d[x,\varepsilon]$, it follows that the topologies generated by the
 quasi-pseudometric $d$ and by the quasi-uniformity $\mathcal U_d$
 agree, i.e., $\, \tau_d=\tau(\mathcal U_d).$

 In this case
 $$
 \mathcal U_d^{-1}=\mathcal U_{\bar d}\quad\mbox{and}\quad  \mathcal U_d^{s}=\mathcal U_{d^s}\,.$$

\begin{remark} Quasi-uniform spaces were studied  by Nachbin starting with 1948 (see  \cite{Nachbin}) as a generalization of uniform spaces introduced by Weil \cite{Weil}. For further developments see \cite{FG} and \cite{kunzi09}.

\end{remark}

 \section{Cauchy sequences and sequential completeness  in quasi-pseudometric and  quasi-uniform spaces}
 The lost of symmetry   causes a lot of trouble in the study of quasi-metric spaces, particularly concerning completeness and compactness.
 Starting from the definition of a Cauchy sequence in a metric space, Reilly et al \cite{reily-subram82} defined 7 kinds of Cauchy sequences, yielding 14 different notions of completeness in quasi-metric spaces, all agreeing with the usual one in the metric case. One of the major drawbacks of most of these  notions is that a  convergent sequence need not  be Cauchy. For a detailed study we refer to the quoted paper by Reilly et al, or to the book \cite{Cobzas}. In the present paper we concentrate on the relations between the corresponding notions of completeness by sequences, nets or filters, as well as to the completeness of the associated quasi-uniform space.

 We give now the definitions following  \cite{reily-subram82}.

 \begin{defi}\label{def.C-seq}
 A sequence $(x_n)$ in $(X,d)$ is called
\begin{itemize}
\item  \emph{left (right) $d$-Cauchy} if for every $\varepsilon > 0$
there exist $x\in X$ and $n_0\in \mathbb{N}$ such that
  $$ d(x,x_n)<\varepsilon
\;\mbox{ (respectively }\; d(x_n,x)<\varepsilon)$$
for all $n\geqslant n_0$;

\item  \emph{$d^s$-Cauchy} if it is   a Cauchy sequence is the pseudometric space $(X,d^s)$, that
is for every $\varepsilon >0$ there exists
$n_0\in \mathbb{N}$ such that
  $$d^s(x_n,x_k)< \varepsilon\; \mbox{ for all }\; n,k\geqslant n_0 \,,$$
 a condition  equivalent to
$$d(x_n,x_k)< \varepsilon \; \mbox{ for all }\; n,k\geqslant n_0\,,$$
as well as to
$$\bar d(x_n,x_k)< \varepsilon \; \mbox{ for all }\; n,k\geqslant n_0\,;$$

\item   \emph{left (right) $K$-Cauchy} if  for every $\varepsilon >0$
there exists $n_0\in \mathbb{N}$ such that
$$ d(x_k,x_n)< \varepsilon \;\mbox{ (respectively }\; d(x_n,x_k)<\varepsilon)$$
for all $n,k\in\mathbb{N}$ with $  n_0\leqslant k\leqslant n$;

\item  \emph{weakly left (right) $K$-Cauchy} if for every $\varepsilon >0$
there exists $n_0\in \mathbb{N}$ such that
$$d(x_{n_0},x_n)< \varepsilon \;\mbox{
(respectively }\; d(x_n,x_{n_0})< \varepsilon)\,, $$
for all $n\geqslant n_0\,.$
\end{itemize}

Sometimes, to emphasize the quasi-pseudometric $d,$ we shall say that a sequence is  left $d$-$K$-Cauchy, etc.
\end{defi}

It seems that  $K$ in the definition of a left
$K$-Cauchy sequence comes from Kelly \cite{kelly63} who considered first this notion.

Some remarks are in order.
 \begin{remark}[\cite{reily-subram82}]\label{p1.Cauchy-seq}
Let $(X,d)$ be a quasi-pseudometric space.
\begin{enumerate}
\item[{\rm 1.}]  These notions are related in the following way: \smallskip

$d^s$-Cauchy \;$\Rightarrow$\; left $K$-Cauchy $\;\Rightarrow\; $ weakly left
$K$-Cauchy $\;\Rightarrow\;$ left  $d$-Cauchy.\smallskip

\noindent  The same implications hold for the corresponding right  notions.
  No one of the above implications is reversible.
\item[{\rm 2.}]  A sequence is left Cauchy (in some sense) with respect to $d$
if and only if it is right Cauchy (in the same sense) with respect to   $\bar{d}.$
\item[{\rm 3.}]    A sequence is $d^s$-Cauchy if and only if it is both left and right
$d$-$K$-Cauchy.
\item[{\rm 4.}] A $d$-convergent sequence is left $d$-Cauchy and a
$\bar d$-convergent sequence is right $d$-Cauchy. For the other notions, a convergent sequence need not be Cauchy.
\item[{\rm 5.}] If each convergent sequence in a regular quasi-metric space $(X,d)$
admits a left $K$-Cauchy subsequence, then $X$ is metrizable
(\cite{kunzi-reily93}).
 \end{enumerate}
 \end{remark}

 We also mention the following simple properties of Cauchy sequences.
 \begin{prop}
 [\cite{bodjana81,reily-vaman84}]\label{p.Cauchy-seq}
 Let $(x_n)$ be a left or right $K$-Cauchy sequence in a quasi-pseudometric space $(X,d).$
 \begin{enumerate}
   \item[{\rm 1.}]  If $(x_n)$    has a subsequence which is $d$-convergent to  $x,$   then $(x_n)$ is $d$-convergent to $x.$
\item[{\rm 2.}]    If $(x_n)$   has a subsequence which is
    ${\bar{d}}$-convergent to  $x,$   then $(x_n)$ is ${\bar{d}}$-convergent to $x.$
\item[{\rm 3.}]      If $(x_n)$ has a subsequence which is $d^s$-convergent to
$x$, then $(x_n)$ is $d^s$-convergent to $x$.
 \end{enumerate}   \end{prop}

To  each of these  notions of Cauchy sequence corresponds two notions
of sequential  completeness,  by asking that the corresponding Cauchy
sequence be $d$-convergent or $d^s$-convergent. Due to the equivalence $d$-left Cauchy $\iff$ $\bar d$-right  Cauchy one obtains nothing new by asking that a $d$-left Cauchy sequence is $\bar d$-convergent. For instance, the $\bar d$-convergence of any left $d$-$K$-Cauchy sequence  is equivalent to  the right $K$-completeness of the space $(X,\bar d).$

\begin{defi}[\cite{reily-subram82}]\label{def.compl-qm}

 A quasi-pseudometric space $(X,d)$ is called:
\begin{itemize}
\item\; \emph{sequentially $d$-complete} if every $d^s$-Cauchy sequence is
 $d$-convergent;

\item\; \emph{sequentially  left $d$-complete} if every left $d$-Cauchy sequence is
 $d$-convergent;

\item\; \emph{sequentially  weakly left (right)  $K$-complete} if every weakly left (right) $K$-Cauchy
 sequence is $d$-convergent;

\item\; \emph{sequentially left (right)  $K$-complete} if every left (right) $K$-Cauchy
 sequence is $d$-convergent;

\item\; \emph{ sequentially left (right)  Smyth  complete} if every left (right) $K$-Cauchy
 sequence is $d^s$-convergent;

\item\; \emph{bicomplete} if the associated pseudometric space
 $(X,d^s)$ is complete, i.e., every $d^s$-Cauchy sequence is
 ${d^s}$-convergent.
 A bicomplete asymmetric normed space $(X,p)$ is called a \emph{biBanach space.}   \end{itemize}
\end{defi}

As we noticed (see  Remark \ref{p1.Cauchy-seq}.4),  each $d$-convergent sequence is left $d$-Cauchy, but
for each of the other notions there are examples of
$d$-convergent sequences that are not Cauchy, which is a major
inconvenience. Another one is that a complete (in some sense) subspace of  a quasi-metric space
need not be  closed.

The implications between these
completeness notions are obtained by reversing the implications between the corresponding notions of Cauchy
sequence from Remark \ref{p1.Cauchy-seq}.1.
\begin{remark}\label{p1.left-complete}
 (a)\;\; These notions of completeness are related in the following way: \smallskip

  sequentially $d$-complete $\;\Rightarrow\;$ sequentially weakly
  left   $K$-complete $\;\Rightarrow\;$  sequentially left   $K$-complete
$\;\Rightarrow\;$  sequentially left   $d$-complete.\smallskip

The same implications hold for the corresponding notions of right completeness.\smallskip

(b)\;\; sequentially left or right Smyth completeness implies bicompleteness.
\end{remark}

No one of the above implication is reversible (see
\cite{reily-subram82}), excepting that  between  weakly
left  and left  $K$-sequential completeness, as it was surprisingly
shown by    Romaguera \cite{romag92}.

\begin{prop}[\cite{romag92}, Proposition 1]\label{p2.left-complete}
A quasi-pseudometric space is sequentially  weakly left $K$-complete  if and only if it is  sequentially  left $K$-complete.  \end{prop}

  A series $\sum_nx_n$  in an  asymmetric seminormed space $(X,p)$  is called \emph{convergent} if there
  exists $x\in X$ such that $x=\lim_{n\to\infty}\sum_{k=1}^nx_k$  (i.e., $\lim_{n\to\infty}p\left(\sum_{k=1}^nx_k-x\right)=0$). The series $\sum_nx_n$  is called \emph{absolutely convergent} if   $ \sum_{n=1}^\infty p(x_n)<\infty.$    It is well-known that a normed space is complete  if and only if every absolutely  convergent series is convergent. A similar result holds in the asymmetric case too.
\begin{prop}[\cite{Cobzas}, Proposition 1.2.6]\label{p.series-complete} Let  $(X,d)$ be a quasi-pseudometric space.
\begin{enumerate}
  \item[{\rm 1.}]  If a sequence $(x_n)$ in  $X$ satisfies
  $ \sum_{n=1}^\infty d(x_n,x_{n+1})$ $<\infty$ \,   $ (\sum_{n=1}^\infty d(x_{n+1},x_{n}) <\infty),$ then it is left (right) $d$-$K$-Cauchy.
  \item[{\rm 2.}]  The quasi-pseudometric space $(X,d)$ is sequentially left (right) $d$-$K$-complete if and only if every sequence $(x_n)$ in $X$ satisfying  $ \sum_{n=1}^\infty d(x_n,x_{n+1}) <\infty$ (resp. $ \sum_{n=1}^\infty d(x_{n+1},x_{n})<\infty)$ is $d$-convergent.
\item[{\rm 3.}] An asymmetric seminormed space $(X,p)$ is sequentially  left $K$-complete if and only if
every absolutely convergent series is convergent.
\end{enumerate}
\end{prop}

\textbf{Cantor type results}

Concerning Cantor-type characterizations of completeness in terms of descending sequences of closed sets (the analog of Theorem \ref{t.Cantor})
we mention the following result.   The \emph{diameter} of a subset $A$
of  a quasi-pseudometric space $(X,d)$ is defined by
\begin{equation}\label{def.diam-qm}
\diam (A) =\sup\{d(x,y) : x,y\in A\}\;.
\end{equation}

It is clear that, as defined,  the diameter is, in fact, the diameter with respect to the associated
pseudometric $d^s.$ Recall that a quasi-pseudometric space is called  sequentially $d$-complete  if every $d^s$-Cauchy sequence is $d$-convergent (see Definition \ref{def.compl-qm}).

\begin{theo}[\cite{reily-subram82}, Theorem 10]\label{t.compl-char}
A quasi-pseudometric space $(X,d)$ is   sequentially $d$-complete if and only if each decreasing sequence
$F_1\supseteq F_2\dots$
of nonempty closed sets with $\diam (F_n)\to 0$ as $n\to \infty$ has nonempty intersection, which is a
singleton if $d$ is a quasi-metric.
\end{theo}

The following characterization of right $K$-completeness was obtained in \cite{chen07}, using a different terminology.
\begin{prop}\label{p1.char-compl}
A quasi-pseudometric space $(X,d)$ is  sequentially right $K$-complete if and only if  any
decreasing sequence of closed $\bar d$-balls
$$B_{\bar d}[x_1,r_1]\supseteq
B_{\bar d}[x_2,r_2]\supseteq\dots \quad\mbox{with}\quad \lim_{n\to\infty}r_n=0\,,$$
 has nonempty intersection.

If the topology $d$ is Hausdorff, then $\bigcap_{n=1}^\infty B_{\bar d}[x_n,r_n]$ contains
exactly one element.\end{prop}

\section{Completeness by nets and filters}
In this section we shall examine the relations between completeness by sequences, nets and filters in quasi-metric spaces. For some notions of completeness they agree, but, as it was shown by Stoltenberg \cite{stolt67}, they can be different for others.
\subsection{Some positive results}
These hold mainly for the notions of left-completeness, and may fail for those of right completeness as we shall see in the next subsection.

The Cauchy properties of a net $(x_i:i\in I)$ in a quasi-pseudometric space $(X,d)$ are defined by analogy with that of sequences, by replacing in Definition \ref{def.C-seq}
the natural numbers with the elements of the directed set $I$.

The situation is good for left Smyth  completeness (see Definition \ref{def.compl-qm}).

\begin{prop}[\cite{romag-tirado15},  Prop. 1]\label{p1.Smyth-compl}
For a quasi-metric space $(X,d)$ the following are equivalent.
\begin{enumerate}
  \item[{\rm 1.}] Every left   $d$-$K$-Cauchy sequence is $d^s$-convergent.
  \item[{\rm 2.}] Every left   $d$-$K$-Cauchy net is $d^s$-convergent.
\end{enumerate} \end{prop}

A quasi-uniform space $(X,\mathcal{U})$ is called \emph{bicomplete} if $(X,\mathcal{U}^s)$ is a complete uniform space. This
notion is useful and easy to handle, because one can appeal to
well known results  from the theory of uniform spaces, but it is not  appropriate for the study of
the specific properties of quasi-uniform spaces, so one introduces adequate definitions, by analogy with quasi-pseudometric spaces.

\begin{defi}\label{def.U-C}
  Let $(X,\mathcal{U})$ be a quasi-uniform space.

   A filter  $\mathcal{F}$ on $(X,\mathcal{U})$ is called:

 \textbullet\; \emph{left} (\emph{right}) $\mathcal{U}$-\emph{Cauchy}   if for every $U\in\mathcal{U}$
 there exists $x\in X$ such that $U(x)\in \mathcal{F}$  (respectively $U^{-1}(x)\in \mathcal{F}$);

 \textbullet\; \emph{left} (\emph{right}) $\mathcal U$-$K$-\emph{Cauchy}  if for every $U\in\mathcal{U}$ there
 exists $F\in\mathcal{F}$ such that  $U(x)\in \mathcal{F}$  (resp. $U^{-1}(x)\in \mathcal{F}$) for all $x\in F$.

 A net $(x_i:i\in I)$ in $(X,\mathcal{U})$ is called:

 \textbullet\; \emph{left $\mathcal{U}$-Cauchy} (\emph{right $\mathcal{U}$-Cauchy}) if for every $U\in\mathcal{U}$
 there exists $x\in X$ and $i_0\in I$ such that $(x,x_i)\in U $  (respectively  $(x_i,x)\in U)$ for all $i\geqslant i_0$;

 \textbullet \; \emph{left  $\mathcal U$-$K$-Cauchy}  (\emph{right  $\mathcal U$-$K$-Cauchy}) if
\begin{equation}\label{def.C-net}
 \quad \forall U\in\mathcal{U},\;   \exists i_0\in I,\; \forall i,j\in I,\; i_0\leqslant i\leqslant j\;\Rightarrow\; (x_i,x_j)\in U \quad\mbox{(resp. \;}(x_j,x_i)\in U \,.\end{equation}
\end{defi}

 The notions of left and right   $\mathcal U$-$K$-Cauchy   filter  were defined by Romaguera in \cite{romag92}.

Observe that
$$
(x_j,x_i)\in U \iff (x_i,x_j)\in U^{-1}\;, $$
so that a filter is right  $\mathcal U$-$K$-Cauchy if and only if it is left  $\mathcal U^{-1}$-$K$-Cauchy.
A similar remark applies to nets.

\begin{defi}
 A quasi-uniform space $(X,\mathcal{U})$ is called:

 \textbullet\; \emph{left $\mathcal{U}$-complete by filters} (\emph{left $K$-complete by filters})  if every
left $\mathcal{U}$-Cauchy (respectively, left $\mathcal U$-$K$-Cauchy)  filter in $X$ is $\tau(\mathcal{U})$-convergent;

\textbullet\; \emph{left $\mathcal{U}$-complete by nets}  (\emph{left  $\mathcal U$-$K$-complete by nets})  if every
left $\mathcal{U}$-Cauchy (respectively, left  $\mathcal U$-$K$-Cauchy)  net in $X$ is $\tau(\mathcal{U})$-convergent;

\textbullet\;    \emph{Smyth left  $\mathcal U$-$K$-complete by nets} if  every left $K$-Cauchy net in
$X$ is $\mathcal{U}^s$-convergent.

The notions of  right completeness are defined similarly, by asking the $\tau(\mathcal{U})$-convergence of
the corresponding right Cauchy filter (or net) with respect to the topology $\tau(\mathcal{U})$
(or with respect to   $\tau(\mathcal{U}^s)$ in the case of Smyth completeness).
\end{defi}

As we have mentioned in Introduction, in  pseudometric spaces the sequential completeness is equivalent to
the completeness defined in terms of filters, or of nets. Romaguera \cite{romag92}  proved  a similar
 result for the left $K$-completeness in  quasi-pseudometric spaces.

\begin{remark} In the case of a quasi-pseudometric space the considered notions take the following form.

 A filter $\mathcal{F}$ in a quasi-pseudometric space $(X,d)$ is called \emph{left} $K$\emph{-Cauchy} if it left $\mathcal U_d$-$K$-Cauchy.
This is equivalent to the fact that for every $\varepsilon >0$ there exists $F_\varepsilon\in\mathcal{F}$ such that
\begin{equation}\label{def.Cauchy-filt}
\forall x\in F_\varepsilon,\quad B_d(x,\varepsilon)\in\mathcal{F}\;.\end{equation}

Also a net $(x_i:i\in I)$  is called \emph{left} $K$-\emph{Cauchy} if  it is left $\mathcal U_d$-$K$-Cauchy or, equivalently, for every $\varepsilon >0$ there exists
$i_0\in I$ such that
\begin{equation}\label{def.Cauchy-net}
\forall i,j\in I,\;\;  i_0\leqslant i\leqslant j \;\Rightarrow \; d(x_i,x_j)<\varepsilon\;.\end{equation}
\end{remark}

\begin{prop}[\cite{romag92}]\label{p.Cauchy-net}
For a quasi-pseudometric space $(X,d)$ the following are equivalent.
\begin{enumerate}
  \item[{\rm 1.}] The space  $(X,d)$ is sequentially left $K$-complete.
 \item[{\rm 2.}]   Every left $K$-Cauchy filter in $X$ is $d$-convergent.
 \item[{\rm 3.}]   Every left $K$-Cauchy net in $X$ is $d$-convergent.
\end{enumerate}\end{prop}

In the case of left $\mathcal U_d$-completeness this equivalence does not hold in general.
\begin{prop}[K\" unzi \cite{kunzi92}]\label{p.left-rho-complete}
A Hausdorff quasi-metric space $(X,d)$ is  sequentially left $d$-complete if and only if
the associated quasi-uniform space $(X,\mathcal{U}_d)$ is left  $\mathcal{U}_d$-complete by filters.
\end{prop}

\subsection{Right $K$-completeness in quasi-pseudometric spaces}
It is  strange  that for the right completeness   the things look worse than for the left completeness.

As remarked Stoltenberg \cite[Example 2.4]{stolt67} a  result similar to Proposition \ref{p.Cauchy-net} does not hold for right $K$-completeness: there exists  a sequentially right $K$-complete $T_1$ quasi-metric space which is not right $K$-complete by nets.  Actually, Stoltenberg \cite{stolt67} proved that the equivalence holds for a more general definition of a right $K$-Cauchy net, but Gregori and Ferrer \cite{greg-fer84} found a gap in Stoltenberg's proof and proposed a new version of Cauchy net. In what follows we will present these results and, in our turn, we shall propose other notion of Cauchy net for which the equivalence holds.

An analog of Proposition \ref{p.Cauchy-net} for right $K$-completeness can   be obtained  only under some supplementary
hypotheses on the quasi-pseudometric space $X$.

A quasi-pseudometric space   $(X,d)$ is called $R_1$ if for all
$x,y\in X,\  d$-$\clos\{x\}\ne d$-$\clos\{y\} $ implies the existence  of
two disjoint $d$-open sets $U,V$ such that $x\in U$ and $y\in V.$
\begin{prop}[\cite{alem-romag97}]\label{p.r-Cauchy-net}
Let $(X,d)$ be a  quasi-pseudometric space. The following are true.
\begin{enumerate}
  \item[{\rm 1.}]   If   $X$ is    right $K$-complete by filters, then every right $K$-Cauchy net in $X$
  is convergent.    In particular,  every right $K$-complete  by filters quasi-pseudometric space
  is  sequentially right $K$-complete.
\item[{\rm 2.}]   If the quasi-pseudometric space $(X,d)$ is $R_1$ then $X$ is  right
$K$-complete by filters if and only if it is sequentially   right $K$-complete.
\end{enumerate} \end{prop}

\textbf{Stoltenberg's example}

As we have mentioned, Stoltenberg \cite[Example 2.4]{stolt67} gave an example of a sequentially right $K$-complete $T_1$ quasi-metric space which is not right $K$-complete by nets, which we shall present now.

Denote by $\mathcal A$ the family of all countable subsets of the interval $[0,\frac 13]$. For $A\in \mathcal A$ let
\begin{align*}
  &X^A_1=A,\;\; X^A_{k+1}=A\cup\left\{\frac 12,\frac 34,\dots,\frac{2^k-1}{2^k}\right\},\; k\in\mathbb{N},\; \mbox{and}\\
  &X^A_{\infty}=A\cup\left\{\frac{2^k-1}{2^k}: k\in\mathbb{N}\right\}=\bigcup\{X^A_k:k\in\mathbb{N}\}\,.
  \end{align*}

Put $\mathcal S=\big\{X^A_k: A\in\mathcal A,\, k\in\mathbb{N}\cup\{\infty\}\big\}$ and define $d:\mathcal S\times\mathcal S\to[0,\infty)$ by
$$
d(X^A_k,X^B_j)=\begin{cases}
  0\quad&\mbox{if}\; A=B\;\mbox{and}\;k=j,\; A,B \in \mathcal A,\, k,j\in\mathbb{N}\cup\{\infty\}\\
 2^{-j}\quad&\mbox{if}\; X^B_j\subsetneqq X^A_k,\;  A,B\in \mathcal A,\, k\in\mathbb{N}\cup\{\infty\},\, j\in\mathbb{N},\\
    1\quad&\mbox{otherwise}\,.
 \end{cases}$$

 \begin{prop}\label{p1.Stolt-ex} $(\mathcal S,d)$ is a sequentially right $K$-complete $T_1$ quasi-metric space which
  is not  right $K$-complete  by nets.
  \end{prop}
  \begin{proof} The proof that $d$ is a $T_1$ quasi-metric on $\mathcal S$ is straightforward.\smallskip

  I.\, $(\mathcal S,d)$ \emph{is  sequentially right} $K$-\emph{complete}.\smallskip

  Let $(X_n)_{n\in\mathbb{N}}$ be a right $K$-Cauchy sequence in $\mathcal S$.
  Then there exists $n_0\in\mathbb{N}$ such that
  $$
  d(X_m,X_n)<1\quad\mbox{for all}\quad m,n\in\mathbb{N} \;\mbox{with}\; n_0\leqslant n\leqslant m\,.$$
For $i\in\mathbb{N}_0=\mathbb{N}\cup\{0\}$ let
  $$
  X_{n_0+i}=X^{A_i}_{k_i}\;\mbox{where}\; A_i\in\mathcal A\;\mbox{and}\; k_i\in\mathbb{N}\cup\{\infty\}\,.$$

  Since
  $$
  d\left(X^{A_{i+1}}_{k_{i+1}},X^{A_i}_{k_i}\right)<1\,,$$
  it follows $k_i\in\mathbb{N}$ for all $i\in\mathbb{N}_0.$ For $0<\varepsilon<1$ there exists $i_0\in\mathbb{N}$ such that

  $$
  d(X_{n_0+i+1},X_{n_0+i})<\varepsilon\quad\mbox{for all}\; i\geqslant i_0\,,$$
  which means that
  $$
 2^{-k_i}= d\left(X^{A_{i+1}}_{k_{i+1}},X^{A_i}_{k_i}\right)<\varepsilon\quad\mbox{for all}\; i\geqslant i_0\,.$$

 This  shows that $\lim_{i\to\infty}k_i=\infty.$

 Let $A=\bigcup\left\{A_i : i\in\mathbb{N}_0\right\}$ and $X=X^A_\infty.$  Then $X^{A_i}_{k_i} \subsetneqq X^A_\infty$, so that
 $$
d(X,X_{n_0+i})=  d\left(X^{A}_{\infty},X^{A_i}_{k_i}\right)= 2^{-k_i}\to 0\;\mbox{ as }\; i\to\infty\,.$$
  which shows that the sequence $(X_n)$ is $d$-convergent to $X$. \smallskip

II.\,\emph{The quasi-metric space $(\mathcal S,d)$ is not  right $K$-complete  by nets.}\smallskip

 Let $\mathcal S_0=\{X^A_k: A\in\mathcal A,\, k\in\mathbb{N}\}$ ordered  by
 $$
 X\leqslant Y\iff X\subseteq Y,\;\mbox{for}\; X,Y\in\mathcal S_0\,.$$

 We have
 $$
 X^A_i\leqslant X^B_j\iff \begin{cases}
   A\subseteq B\quad\mbox{and}\\
   i    \leqslant j
 \end{cases}$$
 for $X^A_i, X^B_j\in\mathcal S_0$, $(\mathcal S_0,\leqslant)$ is directed and the mapping  $\phi:\mathcal S_0\to\mathcal S$ defined by $ \phi(X)=X,\, X\in \mathcal S_0,$ is a net in $\mathcal S.$

 Let us show first that the net $\phi$ is right $K$-Cauchy. For $\varepsilon>0$ choose  $k\in\mathbb{N}$ such that $2^{-k}<\varepsilon.$ For some $C\in\mathcal A,\, X_k^C$ belongs to $\mathcal S_0$  and
$$
d(X^A_j,X^B_i)=2^{-i}\leqslant 2^{-k}<\varepsilon$$
for all
$$
  X^A_j,X^B_i\in\mathcal S_0\;\mbox{ with }\; X^C_k\leqslant X^B_i\leqslant X^A_j,\, X^A_j\ne X^B_i\,,$$
  showing that the net $\phi$ is right $K$-Cauchy.

  Let $X=X_k^C$ be an arbitrary element in $\mathcal S.$ We show that for every $X^A_i\in\mathcal S_0$ there exists $X^B_j\in\mathcal S_0$ with $X^A_i\leqslant X^B_j$
  such that $d(X,X^B_j)=1,$ which will imply that the net $\phi$ is not $d$-convergent to $X$.

  Since $C$ is a countable set,  there exists $x_0\in[0,\frac 13]\setminus C.$ For  an arbitrary $X^A_i\in\mathcal S_0$ let  $B= A\cup\{x_0\}$. Then   $X^B_i\in\mathcal S_0$,
 $X^A_i\leqslant X^B_i$ and $X^C_k\nsubseteq X^B_i$, so that, by the definition of the metric $d$,  $d(X^C_k,X^B_i)=1.$\end{proof}

\textbf{ Stoltenberg-Cauchy nets}

 Stoltenberg \cite{stolt67}   also considered a more general
  definition of a right $K$-Cauchy net as a net $(x_i:i\in I)$   satisfying the condition:
  for every    $\varepsilon>0$ there exists $i_\varepsilon\in I$ such that
\begin{equation}\label{def.Stolt2}
d(x_i,x_j)<\varepsilon\quad\mbox{for all}\quad i,j\geqslant i_\varepsilon \;\mbox{ with }\; i\nleqslant j\,.
\end{equation}

Let us call such a net \emph{Stoltenberg-Cauchy} and  \emph{Stoltenberg completeness} the completeness with respect to Stoltenberg-Cauchy nets.

It follows that, for this definition,
\begin{equation*}\label{eq1.S-Cauchy} d(x_i,x_j)<\varepsilon \quad\mbox{and}\quad d(x_j,x_i)<\varepsilon\quad\mbox{for all }\; i,j\geqslant i_\varepsilon\;\mbox{  with }\; \ i\nsim  j,
\end{equation*}
  where   $i\nsim j$   means that $i,j$ are incomparable  (that is, no one of the relations $i\leqslant j$ or
$j\leqslant i$ holds).\medskip

\textbf{Gregori-Ferrer-Cauchy nets}

Later, Gregori and Ferrer \cite{greg-fer84} found a gap in the proof of Theorem 2.5 from \cite{stolt67} and provided a counterexample to it,  based on Example 2.4 of Stoltenberg (see Proposition \ref{p1.Stolt-ex}).
\begin{example}[\cite{greg-fer84}] \label{ex.Greg-Fer} Let $\mathcal A$,  $(\mathcal S,d)$ be as in the preamble to Proposition \ref{p1.Stolt-ex} and $I=\mathbb{N}\cup\{a,b\}$, where the  set $\mathbb{N}$ is considered with the usual order and $a,b$  are two distinct elements not belonging to $\mathbb{N}$ with
\begin{align*}
  &k\leqslant a,\quad k\leqslant b,\quad\mbox{for all }\;  k\in\mathbb{N},\\
  & a\leqslant a,\quad  b\leqslant b,\quad a\leqslant b,
  \quad b\leqslant a.
\end{align*}
 Consider two sets $A,B\in \mathcal A$ with $A\subsetneqq B$ and let $\phi:I\to \mathcal S$ be given by
$$
\phi(k)=X^A_k,\,k\in\mathbb{N},\quad \phi(a)=X^A_\infty,\quad  \phi(b)=X^B_\infty\,.$$

Then the net  $\phi$ is  right Cauchy  in the sense of \eqref{def.Stolt2}  but not  convergent in   $(\mathcal S,d) $.\end{example}

 Indeed, for $0<\varepsilon<1$ let $k_0\in\mathbb{N}$ be such that $2^{-k_0}<\varepsilon.$

Since $$i\leqslant a,\; i\leqslant b,\; i\geqslant k_0,\;\,\forall i\in\mathbb{N},\; k_0\leqslant a\leqslant b,\;k_0\leqslant  b\leqslant a,$$
it follows that the   condition $i  \nleqslant j$ can hold for some $i,j\in I,\, i,j\geqslant k_0$, in the following cases:

\begin{align*}
  {\rm(a)}\quad &i,j\in\mathbb{N},\;i,j\geqslant k_0,\; j<i;\\
  {\rm(b)}\quad &i=a,\; j\in\mathbb{N},\, j\geqslant k_0\\
  {\rm(c)}\quad &i=b,\; j\in\mathbb{N},\, j\geqslant k_0\\
\end{align*}

  In the case (a), $X^A_j\subsetneqq X^A_i$ and
  $$
d(\phi(i),\phi(j))= d(X^A_i,X^A_j)=2^{-j}\leqslant 2^{-k_0}<\varepsilon\,.$$

 In the case (b), $X^A_j\subsetneqq X^A_\infty$ and again

 $$
d(\phi(a),\phi(j))=d(X^A_\infty,X^A_j)=2^{-j}\leqslant 2^{-k_0}<\varepsilon.$$

The case (c) is similar to (b).

To show that $\phi$ is not convergent let $X\in \mathcal S\setminus \{X^B_\infty\}.$  Then  $b\geqslant i$ for any $i\in I$  and
$$
d(X,\phi(b))=d(X,X^B_\infty)=1\,,$$
so that $\phi$ does not converge to $X$.  If $X=X^B_\infty$, then $a\geqslant i$ for any $i\in I$  and
$$ d(X^B_\infty,\phi(a))=d(X^B_\infty,X^A_\infty)=1\,.$$

Gregori and Ferrer \cite{greg-fer84} proposed a new definition of a right K-Cauchy net, for which  the   equivalence to sequential completeness holds.
\begin{defi}\label{def.GFC}
 A net $(x_i:i\in I)$ in a quasi-metric space $(X,d)$ is  called $GF$-\emph{Cauchy} if one of the following conditions holds:
\begin{enumerate}
\item[\rm(a)]  for every maximal element $j\in I$ the net $(x_i)$ converges to $x_j$;
\item[\rm(b)]  $I$ has no maximal elements and the net $(x_i)$ converges;
\item[\rm(c)]  $I$ has no maximal elements and the net $(x_i)$ satisfies the condition \eqref{def.Stolt2}.
\end{enumerate}
\end{defi}

\textbf{Maximal elements and net convergence}

For a better understanding of this definition we shall analyze the relations between maximal elements in a preordered  set and the convergence of nets. Recall that in the definition of a directed set $(I,\leqslant)$ the relation $\,\leqslant\,$ is supposed to be only a preorder, i.e.  reflexive and transitive and not necessarily antireflexive  (see \cite{Kelley}).
Notice that some authors suppose that in the definition of a directed set  $\,\leqslant\,$ is a partial order (see, e.g., \cite{Willard}).  For a discussion of this matter see \cite[\S 7.12, p. 160]{Schechter}.

  Let  $(I,\leqslant)$ be a preordered set.
  An element $j\in I$  is called:
  \begin{itemize}\item
    \emph{strictly maximal} if there is no  $i\in I\setminus\{j\}$ with $j\leqslant i,$ or, equivalently,
\begin{equation}\label{def1.maxi}
j\leqslant i\;\Rightarrow\; i=j,\quad\mbox{for every }\; i\in I;  \end{equation}
\item  \emph{maximal} if
\begin{equation}\label{def2.maxi}
j\leqslant i\;\Rightarrow\; i\leqslant j,\quad\mbox{for every }\; i\in I\,.  \end{equation}
\end{itemize}

\begin{remark}\label{re1.maxi}  Let $(I,\leqslant)$ be a preordered set.

1. A strictly maximal element is maximal, and if $\leqslant $ is an order, then   these   notions are equivalent.

Suppose now that  the set $I$ is further directed. Then the following hold.

2.    Every  maximal element $j$ of $I$ is a maximum for $I$,  i.e. $i\leqslant j$ for all $i\in I$.

3. If $j$ is a maximal element  and    $j'\in I$  satisfies $j\leqslant j' $, then $j'$ is also a maximal element.

 4.  (Uniqueness of the strictly maximal element) If $j$  is a strictly maximal element, then $j'=j$ for  any maximal element  $j'$  of $I$.
\end{remark}
\begin{proof} 1. These assertions  are obvious.

 2. Indeed, suppose that $j\in I$ satisfies \eqref{def2.maxi}. Then, for arbitrary $i\in I$, there exists $i'\in I$ with $i'\geqslant j,i.$  But
$j\leqslant i' $  implies $i'\leqslant j$ and so $i\leqslant i'\leqslant j.$  (We use the notation $i\geqslant j,k$ for $i\geqslant j\wedge i\geqslant k$.)

3. Let $i\in I$ be such that $j'\leqslant i$. Then $j\leqslant i$ and, by the maximality of $j$, $i\leqslant j\leqslant j'$.

4. If $j$ is strictly maximal and
 $j'$ is a maximal element of $I$, then, by 2,  $j'\le j$ so that, by \eqref{def2.maxi} applied to $j'$,   $j \le j'$  and so, by \eqref{def1.maxi} applied to $j$, $j'=j.$ \end{proof}

We present now some remarks on maximal elements  and   net convergence.
\begin{remark}\label{re.maxi-nets} Let $(X,d)$ be a quasi-metric space, $(I,\leqslant)$ a directed sets and $(x_i:i\in I)$ a net in $X$.

1. If   $(I,\leqslant)$  has a strictly maximal element $j$, then  the net $(x_i)$ is convergent to $x_j$.

 \begin{itemize}\item[\rm 2.(a)] If the net $(x_i)$ converges to $x\in X$, then $d(x,x_j)=0$ for every maximal element $j$ of $I$. If the topology $\tau_d$ is $T_1$ then, further, $x_j=x.$
\item[\rm (b)] If the net $(x_i)$ converges to $x_j$ and to $x_{j'}$, where $j,j'$ are maximal elements of $I$, then $x_j=x_{j'}$.
\item[\rm (c)] If $I$ has maximal  elements and, for some $x\in X$, $x_j=x$ for every maximal element $j$, then  the net $(x_i)$ converges to $x$.
\end{itemize}\end{remark}
\begin{proof}
1. For an arbitrary $\varepsilon>0$ take $i_\varepsilon=j.$ Then   $i\geqslant j$ implies $i=j$, so that
$$
d(x_j,x_i)=d(x_j,x_j)=0<\varepsilon\,.$$

  2.(a) For every $\varepsilon>0$ there exists $i_\varepsilon\in I$ such that $d(x,x_i)<\varepsilon$ for all $i\geqslant i_\varepsilon.$
By Remark \ref{re1.maxi}.2,  $j\geqslant i_\varepsilon$ for every maximal  $j$, so that  $d(x,x_{j})<\varepsilon$ for all $\varepsilon>0$, implying  $d(x,x_{j})=0$.

If the topology $\tau_d$ is $T_1$, then, by  Proposition \ref{p.top-qsm1}.2, $x_j=x$.

(b) By (a), $d(x_{j},x_{j'})=0$ and $d(x_{j'},x_{j})=0$, so that $x_j=x_{j'}$.

(c)
Let $x\in X$ be such that $x_j=x$ for every maximal element $j$ of $I$ and  let $j$ be a fixed maximal element of $I$. For any $\varepsilon>0$ put $i_\varepsilon=j$. Then, by Remark \ref{re1.maxi}.3,  any $i\in I$ such that $i\geqslant j$ is also a maximal element of $I$, so that $x_i=x$ and $d(x,x_i)=0<\varepsilon.$
\end{proof}

Let us say that a quasi-metric space $(X,d)$ is \emph{GF-complete} if every GF-Cauchy net (i.e. satisfying the conditions (a),(b),(c) from Definition \ref{def.GFC}) is convergent. Remark that, with this definition, condition (b) becomes tautological and so superfluous, so it suffices to ask that every net satisfying (a) and (c) be convergent.

By Remarks \ref{re1.maxi}.1 and \ref{re.maxi-nets}.1,  (a) always holds   if $\leqslant$ is an order, so that, in this case, a  net satisfying condition (c) is a  GF-Cauchy net and so GF-completeness agrees with that given by Stoltenberg.\medskip

\textbf{Strongly Stoltenberg-Cauchy nets}

In order to avoid  the shortcomings of the preorder relation,  as, for instance, those put in evidence  by Example \ref{ex.Greg-Fer}, we propose the following definition.
\begin{defi}
  A net $(x_i:i\in I)$ in a quasi-metric space $(X,d)$ is called \emph{strongly Stoltenberg-Cauchy}   if for every $\varepsilon>0$ there exists $i_\varepsilon\in I$ such that,  for all $i,j\geqslant i_\varepsilon$,
 \begin{equation}\label{def.St-Cob}
 (j\leqslant i\vee i\nsim  j)\;\Rightarrow\;d(x_i,x_j)<\varepsilon\,.
 \end{equation}\end{defi}

We present now some remarks on the relations of this notion with  the other notions of Cauchy net (Stoltenberg and GF), as well as the relations between the corresponding completeness notions. It is obvious that in the case of a sequence $(x_k)_{k\in\mathbb{N}}$ each of these three notions agrees with the right $K$-Cauchyness of $(x_k)$.

\begin{remark}\label{re.St-Cob} Let $(x_i:i\in I)$ be a net in a quasi-metric space $(X,d)$.

1.(a)  We have
\begin{equation}\label{eq.Stolt-str}
i\nleqslant j\;\Rightarrow\;\big((j\leqslant i\wedge i\ne j)\vee (i\nsim  j)\big)\,,\end{equation}
for all $i,j\in I$. If $\leqslant$ is an order, then the reverse implication also holds.

(b) If the net  $(x_i:i\in I)$     satisfies \eqref{def.St-Cob} then it
  satisfies \eqref{def.Stolt2}, i.e. every strong Stoltenberg-Cauchy net is Stoltenberg-Cauchy.   If $\leqslant $ is an order, then these notions are equivalent.

  Hence, net completeness with respect to
\eqref{def.Stolt2} (i.e. Stoltenberg completeness) implies net completeness with respect to   \eqref{def.St-Cob};

2. Suppose that the net  $(x_i:i\in I)$     satisfies \eqref{def.St-Cob}.

  (a) If $j,j'$ are maximal elements of $I$,  then $x_j=x_{j'}$. Hence,
  if $I$ has maximal elements,   then there exists $x\in X$ such that $x_j=x$ for every maximal element $j$ of $I$, and the net $(x_i)$ converges to $x.$

(b) Consequently,   the net $(x_i)$  also
  satisfies the conditions (a) and (c) from Definition \ref{def.GFC}, so that, GF-completeness implies  completeness with respect to   \eqref{def.St-Cob}.
\end{remark}\begin{proof}

  1.(a)  Let $i,j\in I$ with $i\nleqslant j$. Since   $j\leqslant i$ and $i\ne j$ if $i,j$ are comparable,  the implication \eqref{eq.Stolt-str} holds. If $\leqslant$ is an order and $i\ne j$, then
 $j\leqslant i\;\Rightarrow\; i\nleqslant j$ and  $i\nsim j\;\Rightarrow\; i\nleqslant j$.

 (b) Since it suffices to ask that \eqref{def.Stolt2} and \eqref{def.St-Cob} hold only for distinct $i,j\geqslant i_\varepsilon$, the equivalence of these notions in the case when $\leqslant$ is an order follows.

Suppose that the net $(x_i)$ satisfies \eqref{def.St-Cob}. For $\varepsilon>0$ choose $i_\varepsilon\in I$ according to \eqref{def.St-Cob} and let $i,j\geqslant i_\varepsilon$ with $i\nleqslant j.$
Taking into account \eqref{eq.Stolt-str}  it follows $d(x_i,x_j)<\varepsilon$, i.e.  $(x_i)$ satisfies \eqref{def.Stolt2}.

  Suppose now that every net satisfying \eqref{def.Stolt2} converges and let $(x_i)$ be a net in $X$ satisfying \eqref{def.St-Cob}. Then it satisfies \eqref{def.Stolt2} so it converges.

  2.(a) Let $j,j'$ be maximal elements of $I$.  For $\varepsilon >0$ choose $i_\varepsilon$ according to \eqref{def.St-Cob}. By Remark \ref{re1.maxi}.2, $j,j'\geqslant i_\varepsilon,\, j\leqslant j',\ j'\leqslant j$, so that $d(x_{j'},x_j)<\varepsilon$ and $d(x_{j},x_{j'})<\varepsilon$. Since these inequalities hold for every $\varepsilon >0$, it follows $d(x_{j'},x_j)=0=d(x_{j},x_{j'})$ and so $x_j=x_{j'}$.  The convergence of the net $(x_i)$ follows from Remark \ref{re.maxi-nets}.2.(c).

(b) The assertions on  GF-Cauchy nets  follow from (a).
\end{proof}

The following example shows that the notion of strong Stoltenberg-Cauchy net is effectively stronger that that of Stoltenberg-Cauchy net.
\begin{example} Let $X=\mathbb{R}$ and $u(x)=x^+,\, x\in X$ be the asymmetric norm  defined in Example \ref{ex.u-topR}. Then $d_u(x,y)=u(y-x)=(y-x)^+,\, x,y\in X,$ is a quasi-metric on $X$. Let $I=\mathbb{N}\cup\{a,b\}$ be the directed set considered in Example \ref{ex.Greg-Fer}. Define $x_k=0$ for $k\in\mathbb{N},\, x_a=1$ and $x_b=2$. Then $(x_i:i\in I)$ is Stoltenberg-Cauchy but not strongly  Stoltenberg-Cauchy nor GF-Cauchy.
 \end{example}

 Indeed, for $\varepsilon >0$ let $i_\varepsilon =1$ and $i,j\geqslant 1$ with $i\nleqslant j$. Then  $j\in \mathbb{N}$ and  we distinguish three possibilities:
 \begin{itemize}
   \item \; $i\in \mathbb{N}$  \, and $ j<i\,,$ when $d_u(x_i,x_j)=(x_j-x_i)^+=0$;
   \item    $i=a $   and  $d_u(x_a,x_j)=(x_j-x_a)^+=(0-1)^+=0$;
    \item    $i=b $   and  $d_u(x_b,x_j)=(x_j-x_b)^+=(0-2)^+=0$.
 \end{itemize}

It follows that  $d_u(x_i,x_j)=0<\varepsilon$   in all cases, showing that  $(x_i)$ satisfies the condition \eqref{def.Stolt2}, that is, it is Stoltenberg-Cauchy.

Notice that any two elements in $I$ are comparable. Let $0<\varepsilon<1.$ Since, for every $i_\varepsilon \in I, \, a,b\geqslant i_\varepsilon$ and $\ b\leqslant a$,
but $d_u(x_a,x_b)=(2-1)^+=1$, it follows  that \eqref{def.St-Cob} fails, that is, $(x_i)$ is not  strongly Stoltenberg-Cauchy.

Since $a$ is a maximal element  of $I$,   $b\geqslant i_\varepsilon$ for any $  i_\varepsilon\in I$, the above equality ($d_u(x_a,x_b)=1$) shows that the net $(x_i)$ does not converge to $x_a$. Consequently $(x_i)$ is not GF-Cauchy (see Definition \ref{def.GFC}).

We show now  that completeness by nets with respect to \eqref{def.St-Cob} is equivalent to sequential right $K$-completeness.
\begin{prop}[\cite{stolt67}, Theorem 2.5]\label{p.Stolt-seq-compl}
  A $T_1$ quasi-metric space $(X,d)$ is sequentially  right $K$-complete if and only if every
 net in $X$ satisfying \eqref{def.St-Cob} is $d$-convergent, i.e. every strongly Stoltenberg-Cauchy net is convergent.
  \end{prop}\begin{proof}
    We have only to prove that  the sequential right  $K$-completeness implies that   every
 net in $X$ satisfying \eqref{def.St-Cob} is $d$-convergent.

    Let $(x_i:i\in I)$ be   a  net in $X$ satisfying \eqref{def.St-Cob}. Let   $i_k\geqslant i_{k-1},\ k\geqslant 2,$
 be   such that \eqref{def.St-Cob} holds for $\varepsilon =1/2^k,\ k\in \mathbb{N}.$

 This is possible. Indeed, take $i_1$ such that \eqref{def.St-Cob} holds for $\varepsilon=1/2.$  If $i'_2$ is such that \eqref{def.St-Cob} holds for $\varepsilon=2^{-2}$, then pick
 $i_2\in I$ such that $i_2 \geqslant i_1,i'_2.$ Continuing by induction one obtains the desired sequence $(i_k)_{k\in\mathbb{N}}.$

We distinguish two cases.\smallskip

\emph{Case I.} \; $\exists j_0\in I,\;\exists k_0\in\mathbb{N},\quad \forall k\geqslant k_0,\; i_k\leqslant j_0\,.$\smallskip

Let $i\geqslant j_0.$ Then for every $k$,\,  $i_k\leqslant j_0\leqslant i$  implies  $d(x_i,x_{j_0})<2^{-k}$
so that  $d(x_i,x_{j_0})=0.$ Since the quasi-metric space $(X,d)$ is $T_1,$ it follows
$x_i=x_{j_0}$ for all $i\geqslant j_0$ (see Proposition \ref{p.top-qsm1}), so that the net
$(x_i:i\in I)$ is $d$-convergent to $x_{j_0}.$\smallskip

\emph{Case II.} \; $ \forall j\in I,\; \forall k\in\mathbb{N},\; \exists k'\geqslant k,\quad i_{k'}\nleqslant j\,.$\smallskip

 The inequalities $d(x_{i_{k+1}},x_{i_{k}})<2^{-k}, \ k\in \mathbb{N},$  imply that the sequence
    $(x_{i_k})_{k\in\mathbb{N}} $ is right $K$-Cauchy  (see Proposition \ref{p.series-complete}), so it is $d$-convergent to some $x\in X.$

For
  $\varepsilon > 0$ choose  $k_0\in\mathbb{N} $  such that   $2^{-k_0}<\varepsilon$ and $d(x,x_{i_k})<\varepsilon $
for  all $k\geqslant k_0.$

Let $i\in I,\, i\geqslant i_{k_0}.$ By hypothesis, there exists $k\geqslant k_0$ such that $i_k\nleqslant i,$ implying $i\leqslant i_k\vee i_k\nsim i.$
Since $i_{k_0}\leqslant i_k,i$, by the choice of $i_{k_0}$, $d(x_{i_k},x_i)<2^{-k_0}<\varepsilon$  in both of these cases. But then
$$
d(x,x_i)\leqslant d(x,x_{i_k})+d(x_{i_k},x_i)<2\varepsilon\,,$$
proving the convergence of the net  $(x_i)$ to $x.$
\end{proof}

\textbf{The proof of Proposition \ref{p.Stolt-seq-compl} in the case of GF-completeness} \smallskip

As the result in  \cite{greg-fer84} is given without proof, we shall supply one.
\begin{prop} A $T_1$ quasi-metric space $(X,d)$ is right $K$-sequentially complete
if and only if every net satisfying the conditions (a) and (c) from Definition \ref{def.GFC} is convergent.\end{prop}
\begin{proof}
 Obviously, a proof is needed only  for the case (c).

Suppose that the directed set $(I,\leqslant)$ has no maximal elements and let $(x_i:i\in I)$ be a net in a quasi-metric space $(X,d)$ satisfying \eqref{def.Stolt2}.

 The proof   follows  the  ideas of the proof of Proposition \ref{p.Stolt-seq-compl} with some further details. Let  $i_k\leqslant i_{k+1},\, k\in\mathbb{N},$ be a sequence of indices in $I$ such that
$d(x_i,x_j)<2^{-k}$ for all $i,j\geqslant i_k$ with $i\nleqslant j$. We show that we can further suppose  that $i_{k+1}\nleqslant i_{k}$.

Indeed, the fact that $I$ has no maximal elements implies that for every $i\in I$ there exists $i'\in I$ such that
\begin{equation}\label{eq1.no-maxi}
i\leqslant i'\;\mbox{ and }\;i'\nleqslant i.
\end{equation}

Let $i'_1\in I$ be such that \eqref{def.Stolt2} holds for $\varepsilon=2^{-1}.$ Take $i_1$ such that $i'_1\leqslant i_1$ and $i_1\nleqslant i'_1$. Let $i_2'\geqslant i_1$ be such that
\eqref{def.Stolt2} holds for $\varepsilon=2^{-2}$  and let $i_2\in I$ satisfying $i_2'\leqslant i_2$ and $i_2\nleqslant i_2'$. Then $i_1\leqslant i_2$ and $i_2\nleqslant i_1$, because
 $i_2\leqslant i_1\leqslant i'_2$ would  contradict the choice of $i_2.$

By induction one obtains a sequence $(i_k)$ in $I$ satisfying $i_k\leqslant i_{k+1}$ and $i_{k+1}\nleqslant i_k$ such that \eqref{def.Stolt2} is satisfied with
$\varepsilon=2^{-k}$  for every $i_k$.

We shall again consider  two cases.\smallskip

\emph{Case I.} \; $\exists j_0\in I,\;\exists k_0\in\mathbb{N},\quad \forall k\geqslant k_0,\; i_k\leqslant j_0\,.$\smallskip

Let $i\geqslant j_{0}$. By \eqref{eq1.no-maxi} there exists $i'\in I$ such that $i\leqslant  i'$ and $i'\nleqslant i,$ implying $d(x_{i'},x_i)<2^{-k}$ for all $k\geqslant k_0$,
that is $d(x_{i'},x_i)=0$, so that, by  $T_1$, $x_{i'}=x_i$.

We also have $i'\nleqslant j_0$ because   $i'\leqslant j_0$  would imply  $i'\leqslant i$, in contradiction to the choice of $i'$. But then, $d(x_{i'},x_{j_0})<2^{-k}$ for all $k\geqslant k_0$,
so that, as above,  $d(x_{i'},x_i)=0$ and  $x_{i'}=x_{j_0}$.

Consequently, $x_i=x_{j_0}$ for every $i\geqslant j_0$, proving the convergence of the net $(x_i)$ to $x_{j_0}$.\smallskip

\emph{Case II.} \; $ \forall j\in I,\; \forall k\in\mathbb{N},\; \exists k'\geqslant k,\quad i_{k'}\nleqslant j\,.$\smallskip

The condition $d(x_{i_k},x_{i_{k+1}})<2^{-k}, \, k\in\mathbb{N},$  implies that the sequence $(x_{i_k})_{k\in\mathbb{N}}$ is right $K$-Cauchy, so that there exists $x\in X$ with $d(x,x_{i_k})\to 0$ as $k\to \infty$.

For $\varepsilon>0$ let $k_0\in\mathbb{N}$ be such that $2^{-k_0}<\varepsilon$ and $d(x,x_{i_k})<\varepsilon $ for all $k\geqslant k_0$.

 Let $i\geqslant i_{k_0}$. By II, for $j=i$ and $k=k_0$, there exists $k\geqslant k_0 $ such that $i_k\nleqslant i$. The conditions  $k\geqslant k_0,\, i_{k_0}\leqslant i,\, i_{k_0}\leqslant i_k$ and $i_k\nleqslant i$
 imply
 $$
 d(x,x_{i_k})<\varepsilon\;\mbox{ and }\; d(x_{i_k},x_i)<\varepsilon\,,$$
 so that
 $$
 d(x,x_i)\leqslant d(x,x_{i_k})+d(x_{i_k},x_i)<2\varepsilon,$$
 for all $i\geqslant i_{k_0}$.
\end{proof}

\section{Conclusions} The present paper shows that there are big differences between the notions of completeness in metric and  in quasi-metric spaces, but in spite of this, by giving appropriate definitions we can make the things to look better.  The differences are even bigger in what concerns compactness in these spaces -- in contrast to the metric case, in quasi-metric spaces the notions of compactness, sequential compactness  and countable compactness can be different (see \cite{Cobzas}).

\textbf{Acknowledgements.}  The author expresses his deep gratitude to reviewers for the careful reading of the manuscript and for the suggestions that led
to  improvements   both in presentation and in contents.

\end{document}